\documentclass[english]{amsart}

\usepackage{graphicx}
\usepackage{amsmath,amsfonts,amsthm}
\usepackage{amssymb}
\usepackage{multirow,array}
\usepackage{comment}
\usepackage{enumitem}
\usepackage[utf8]{inputenc}  
\usepackage[T1]{fontenc}
\usepackage{emptypage}
\usepackage{hyperref}
\usepackage{color}
\usepackage{booktabs}
\usepackage{csquotes}

\numberwithin{equation}{section}

\newtheorem{theorem}{Theorem}[section]

\newtheorem{corollary}[theorem]{Corollary}

\newtheorem{definition}[theorem]{Definition}

\newtheorem{lemma}[theorem]{Lemma}

\newtheorem{proposition}[theorem]{Proposition}

\theoremstyle{remark}
\newtheorem{remark}[theorem]{Remark}

\newcommand{\ExpSig}{\textnormal{ESig}}

\newcommand{\OO}{\mathcal{O}}

\newcommand{\TT}{\mathcal{T}}
\newcommand{\CC}{\mathcal{C}}
\newcommand{\PP}{\mathcal{P}}

\newcommand{\HH}{\mathcal{H}}
\newcommand{\FF}{\mathcal{F}}
\newcommand{\GG}{\mathcal{G}}
\newcommand{\BB}{\mathcal{B}}
\newcommand{\LL}{\mathcal{L}}

\newcommand{\SSS}{\mathcal{S}}
\newcommand{\A}{\mathcal{A}}

\newcommand{\R}{\mathbb R}
\newcommand{\C}{\mathbb C}

\newcommand{\X}{\mathbf X}
\newcommand{\Y}{\mathbf Y}

\newcommand{\var}{\textnormal{-var}}

\newcommand{\EEE}[1]{\mathbb E \left[ #1 \right]}

\newcommand{\spn}[1]{\textnormal{span}\left( #1 \right)}

\newcommand{\1}{\mathbf{1}}

\newcommand{\LLL}{\mathbf L}

\newcommand{\uu}{\mathfrak u}

\newcommand{\Tfrak}{\mathfrak T}

\newcommand{\gen}[1]{\langle #1 \rangle}
\newcommand{\floor}[1]{\lfloor #1 \rfloor}

\newcommand\restr[2]{{\left.\kern-\nulldelimiterspace 
  #1 
  \vphantom{\big|} 
  \right|_{#2} 
  }}

\makeatother

\usepackage[english]{babel}

\begin{document}

\title{An isomorphism between branched and geometric rough paths}

\author{Horatio Boedihardjo}
\address{H. Boedihardjo,
Department of Mathematics and Statistics,
University of Reading,
Reading RG6 6AX,
United Kingdom}
\email{h.s.boedihardjo@reading.ac.uk}

\author{Ilya Chevyrev}
\address{I. Chevyrev,
Mathematical Institute,
University of Oxford,
Andrew Wiles Building,
Radcliffe Observatory Quarter,
Woodstock Road,
Oxford OX2 6GG,
United Kingdom}
\email{chevyrev@maths.ox.ac.uk}



\begin{abstract}
We exhibit an explicit natural isomorphism between spaces of branched and geometric rough paths.
This provides a multi-level generalisation of the isomorphism of Lejay--Victoir~\cite{LV06} as well as a canonical version of the It{\^o}--Stratonovich correction formula of Hairer--Kelly~\cite{HairerKelly15}.
Our construction is elementary and uses the property that the Grossman--Larson algebra is isomorphic to a tensor algebra.

We apply this isomorphism to study signatures of branched rough paths.
Namely, we show that the signature of a branched rough path is trivial if and only if the path is tree-like, and construct a non-commutative Fourier transform for probability measures on signatures of branched rough paths.
We use the latter to provide sufficient conditions for a random signature to be determined by its expected value, thus giving an answer to the uniqueness moment problem for branched rough paths.
\end{abstract}

\maketitle

\tableofcontents

\section{Introduction}

Differential equations of the form
\begin{equation}
\mathrm{d}Y_{t}=\sum_{i=1}^{d}f_{i}(Y_{t})\mathrm{d}X_{t}^{i}\;, \quad Y_{0}=y \in \R^e\;,\label{eq:SDE}
\end{equation}
where $f_{i}:\mathbb{R}^{e}\rightarrow\mathbb{R}^{e}$ are smooth
vector fields and $X=(X^{1},\ldots,X^{d}):[0,1]\rightarrow\mathbb{R}^{d}$
is a path, are ubiquitous in control theory and stochastic analysis.
A major difficulty
in giving meaning to~\eqref{eq:SDE} in the stochastic setting is that many examples of $X$ (e.g., the sample paths of Brownian motion) are
highly irregular, and beyond the reach of Lebesgue--Stieltjes--Young integration theory.

A key observation of Lyons~\cite{Lyons98} in his introduction of rough paths theory was that equation~\eqref{eq:SDE} is well-defined in a pathwise sense
as long as the driver $X$ is a {\it geometric rough path}.
One of the successes of rough paths theory has been a separation of probabilistic and deterministic arguments required to solve such differential equations, which has allowed the driver $X$ to be taken well beyond semi-martingale theory (e.g., fractional Brownian motion with Hurst parameter $H > \frac{1}{4}$) as well as given a novel perspective on many classical results in stochastic analysis.

The geometric theory of rough paths is
restricted to drivers which respect the usual chain rule of smooth paths.
A meaningful generalisation of the theory was provided by Gubinelli's notion of branched rough paths~\cite{Gubinelli10}.
Branched rough paths can be seen as an adaptation of Butcher series~\cite{Butcher16}, which are used in the analysis of non-linear ODEs, to the setting of rough driving paths.

Branched rough paths also fall naturally into Hairer's theory of regularity structures~\cite{Hairer14}, in which drivers of singular stochastic PDEs are typically ``non-geometric'' due to renormalisation procedures required to obtain non-trivial limits.
A careful analysis of branched rough paths in this framework was recently carried out in~\cite{BCFP17}, 
which in turn inspired an algebraic method to renormalise SPDEs in regularity structures~\cite{BCCH17}.
See also~\cite{BFGMS17} for a recent application of techniques from regularity structures in a setting very close to branched rough paths. 

There have been several attempts to reformulate non-geometric as geometric rough paths (the inclusion of geometric into branched/non-geometric rough paths is clear, but a canonical reverse inclusion is not).
For $p\in(2,3)$, it was shown in~\cite{LV06} that the space of non-geometric $p$-rough paths is canonically isomorphic to a space of geometric rough paths over a larger space with mixed $(p,p/2)$-regularity.
This led to a natural (deterministic) It{\^o} formula for rough differential equations.
Another approach was taken in~\cite{HairerKelly15}, where the authors showed that to every branched $p$-rough path one can associate a geometric $p$-rough path, also leading to an It{\^o}-type formula.
The construction of~\cite{HairerKelly15}, however, is highly non-canonical (relying on the Lyons--Victoir extension theorem~\cite{LyonsVictoir07}) and requires a non-trivial transformation of the original branched rough path.
It therefore comes short of the satisfactory answer provided in~\cite{LV06} for $p\in (2,3)$.

The first main contribution of this article is to exhibit, for any $p \geq 1$, an explicit natural isomorphism between the space of branched $p$-rough paths and a suitable space of geometric rough paths, see Theorem~\ref{thm:branched are geometric}.
This result can be seen as a multi-level generalisation of the result of~\cite{LV06}, which is important for studying rough signals with $3$-variation (or $\frac{1}{3}$-H{\"o}lder) regularity and worse.
In turn, this isomorphism gives rise to a canonical It{\^o}--Stratonovich correction formula, see Proposition~\ref{prop:Ito-Strat correction}.

Our construction is elementary and relies on a result of Foissy~\cite{Foissy02} and Chapoton~\cite{Chapoton10}, which was already conjectured by Agrachev--Gamkrelidze~\cite{AG81}, that the Grossman--Larson Hopf algebra, in which branched rough paths take values, is isomorphic to a tensor Hopf algebra.
Theorem~\ref{thm:branched are geometric} follows by combining the result of Foissy--Chapoton with a natural way to measure mixed regularity (we find it convenient for this purpose to work with the notion of $\Pi$-rough paths~\cite{Gyurko16}).

The second contribution of this article is in the study of signatures of branched rough paths.
By reducing to the geometric setting, we show that
\begin{enumerate}[label=(\alph*)]
\item \label{point:sig unique} (Theorem~\ref{thm:unique}) the signature map uniquely determines branched rough paths up to tree-like equivalence,
\item \label{point:Fourier} (Theorem~\ref{thm:charFunc}) there exists a Fourier transform on probability measures on signatures of branched rough paths which uniquely determines measures.
\end{enumerate}
As a consequence of the final point, we give sufficient conditions for the expected value of the signature of a random branched rough path to uniquely determine its law, see Proposition~\ref{prop:moment prob}.

The signature (also known as the Chen–Fliess series)  can be seen as the exponential function on path space and arose in Chen's study of the cohomology of loop space~\cite{Chen57}.
It was shown in~\cite{Hambly10} that the signature uniquely determines bounded variation paths up to tree-like equivalence, and this result was recently generalised to all weakly geometric rough paths~\cite{Boedihardjo16}.

In the stochastic setting, it was shown in~\cite{ChevyrevLyons16} that there exists a natural non-commutative Fourier transform (or characteristic function) on random signatures of geometric rough paths, and that, for of a wide class of stochastic processes, the expectation of the signature uniquely determines its law.
The terms in the expected signature also arise naturally in stochastic Taylor expansions~\cite{Azencott82,BenArous89,FV08} and were instrumental in the development of cubature methods on Wiener space~\cite{LyonsVictoir}.

The difficulty in proving~\ref{point:sig unique} and~\ref{point:Fourier} for branched rough paths directly is that the techniques developed in~\cite{Boedihardjo16, ChevyrevLyons16} rely on the geometric nature of the rough paths at hand.
In particular, as we explain in Section~\ref{subsec:char funcs}, one defines a Fourier transform in terms of unitary representations of the signature group, and the study of these representations in~\cite{ChevyrevLyons16} relies crucially on the fact that the signature is a group-like element of a tensor Hopf algebra.
That said, the reduction to the geometric setting greatly simplifies (though does not trivialise) both problems.

\begin{remark}
We work in this paper only with rough paths over a finite-dimensional vector space.
We note that a framework to address branched rough paths over general Banach spaces has been given in~\cite{CassWeidner16}, and we expect that the isomorphism and uniqueness of signature from Theorems~\ref{thm:branched are geometric} and~\ref{thm:unique} carry over to this setting; the construction of a Fourier transform however is more delicate (in~\cite{ChevyrevLyons16}, this was not achieved for geometric rough paths over infinite-dimensional Banach spaces).
\end{remark}

The structure of the article is the following.
In Section~\ref{sec:prelims} we collect some preliminary material and recall the result of Foissy--Chapoton.
In Section~\ref{sec:Pi rough paths} we recall and prove several properties of geometric $\Pi$-rough paths.
In Section~\ref{sec:branched are geo} we make explicit the isomorphism between branched and geometric rough paths, and derive several important consequences.
In Section~\ref{sec:sigs of branched} we prove the aforementioned results concerning signatures of branched rough paths.
We conclude in Section~\ref{subsec:example} with an explicit example of the isomorphism from Theorem~\ref{thm:branched are geometric} applied to lifts of semi-martingales, and show how to verify the conditions of Proposition~\ref{prop:moment prob} to solve the moment problem.

\medskip
{\bf Acknowledgements.}
I.C. is funded by a Junior Research Fellowship of St John's College, Oxford.
H.B. would like to thank Atul Shekhar for the useful discussions.
We would like to thank the anonymous referee for useful suggestions and pointing out the reference~\cite{AG81}.

\section{Preliminaries}
\label{sec:prelims}

\subsection{Notation}
Throughout the paper, we fix $p \geq 1$.
For a vector space $V$, we let $T((V)) = \prod_{m = 0}^\infty V^{\otimes m}$ \label{T((V)) page ref} denote the Hopf algebra of formal series in tensors of $V$.
By convention, we set $V^{\otimes 0} = \spn{\1} \simeq \R$, where $\1$ is the unit of $T((V))$.
We denote the product and coproduct on $T((V))$ by $\dot \otimes$ \label{dot otimes page ref} and $\Delta$ \label{Delta page ref} respectively.
Let $T(V)= \bigoplus_{m = 0}^\infty V^{\otimes m}$ \label{T(V) page ref} denote the (free) tensor Hopf algebra generated by $V$, which we identify with the subspace of $T((V))$ consisting of finite series.

We denote by $|\cdot|$ a norm on a vector space. 
Whenever we refer to a norm on a tensor product $V_1\otimes\ldots\otimes V_m$, \label{otimes page ref} where $V_1,\ldots, V_m$ are normed spaces, we shall always assume this is the projective norm.
We then let $V_1\hat\otimes\ldots\hat\otimes V_m$ \label{hat otimes page ref} denote its completion.
For a locally convex space $V$, we consequently define the algebra $P(V) = \prod_{m = 0}^\infty V^{\hat\otimes m}$ \label{P(V) page ref}.
Note that $T((V))$ is a subspace of $P(V)$, and that $T((V))=P(V)$ whenever $V$ is finite dimensional.

\subsection{Branched rough paths and the GL algebra}
\label{subsec:branched rps}
We briefly recall the notion of a branched $p$-rough path over $\R^d$.
For a graph $\tau$, we let $N_\tau$ denote its node set.\label{N_{tau} page ref}
A \textit{labelled rooted tree} (henceforth simply \textit{tree}) is a triple $(\tau,r,\LL)$, where $\tau$ is a graph which contains no loops, $r\in N_\tau$ is a distinguished vertex called the root, and $\LL : N_\tau \to \{1,\ldots, d\}$ \label{LL page ref} is a labelling of the nodes of $\tau$.
A forest is an (unordered) multiset of trees;
by convention, we postulate the empty set, denoted by $\1$,
to be a forest.
We will denote by $\mathcal{F}$ \label{FF page ref} and $\mathcal{T}$ \label{TT page ref} the
set of all forests and trees respectively.
Let $\mathcal{B}$ \label{BB page ref} and $\HH$\label{HH page ref}
denote the formal linear span over $\R$ of $\mathcal{T}$ and $\mathcal{F}$ respectively.
Let $\HH^* = \{\sum_{\sigma \in \FF} \lambda_\sigma\sigma : \lambda_\sigma \in \R \; \forall \sigma \in \FF\}$ \label{HH^* page ref}denote the space of formal series in forests. The space $\HH$ contains exactly the series in $\HH^*$ where all but finitely many terms are zero. 

We equip both $\HH$ and $\HH^*$ with the structure of the \textit{Grossman--Larson Hopf algebra} and denote by $\star$ \label{star page ref} and $\delta$ \label{delta page ref} the corresponding product and coproduct respectively (note that $\1$ is the unit for $\star$).
For our purposes, we do not need the precise definitions of $\star$ and $\delta$ and refer the reader to~\cite[Sec.~3]{BCFP17} and~\cite[Sec.~2]{HairerKelly15} for details.

The \textit{Butcher group} $\GG^*$ \label{GG^* page ref} is defined to be the set
of all group-like elements of $\mathcal{H}^*$, i.e., $\GG^* = \{g \in \HH^* : \delta(g)=g\otimes g,\;g\neq 0\}$.
One can check that $\GG^*$ indeed forms a group.

For $\sigma\in\mathcal{F}$, let $|\sigma|$ denote the number of
nodes in $\sigma$. \label{|tau| page ref}
For $N \geq 0$, note that the space of series of the form $\sum_{|\sigma| > N} \lambda_\sigma \sigma$ is an ideal of $\HH^*$.
We can thus form the corresponding quotient algebra $\HH^N$ \label{HH^N page ref} and denote by $\rho^{N}: \HH^* \to \HH^N$ \label{rho page ref} the canonical projection.
By construction, the (finite) set $\FF^N := \{\sigma \in \FF : |\sigma| \leq N \}$ is a basis for $\HH^N$ and we equip $\HH^N$ with the inner product for which $\FF^N$ is an orthonormal basis.
We define the \textit{level-$N$ Butcher group} $\mathcal{G}^{N}$ \label{GG^N page ref} as the image of $\GG^*$ under $\rho^{N}$.

We also denote by $\HH^{(N)}$ \label{HH^(N) page ref} the space spanned by $\FF^{(N)} := \{\sigma \in \FF : |\sigma|=N\}$ and let $\pi^N$ \label{pi^N page ref} denote the projection $\pi^N : \sum_{\sigma} \lambda_\sigma \sigma \mapsto \sum_{|\sigma|=N} \lambda_\sigma \sigma$, where the series $\sum_{\sigma} \lambda_\sigma\sigma$ is understood as an element of $\HH^*$ or $\HH^{M}$ for some $M \geq N$.

Define the simplex $\triangle=\{(s,t):0\leq s\leq t\leq1\}$. \label{triangle page ref}
A \textit{control} is a continuous function $\omega:\triangle\rightarrow [0,\infty)$ \label{omega page ref}
such that 
\[
\omega(s,u)+\omega(u,t)\leq\omega(s,t)\;, \quad \forall \; (s,u)\;,(u,t) \in \triangle \;.
\]
\begin{definition}
A \textit{branched $p$-rough path} is a continuous function $X:\triangle\rightarrow\mathcal{G}^{\lfloor p\rfloor}$
such that
\begin{enumerate}
\item for all $s\leq u\leq t$
\[
X_{s,u}\star X_{u,t}=X_{s,t}\;;
\]
\item there exists a control $\omega$ such that for all $n=1,\ldots, \floor p$ and $s\leq t$, 
\[
|\pi^n X_{s,t}|\leq\omega(s,t)^{n/p}\;.
\]
\end{enumerate}
\end{definition}

We recall the following extension theorem.

\begin{theorem}[\cite{Gubinelli10}]
Let $X:\triangle\rightarrow \GG^{\lfloor p\rfloor}$ be a branched $p$-rough path. Then there exists a unique extension $S(X):\triangle\rightarrow \mathcal{H}^*$ \label{S(X) page ref} such that
\begin{enumerate}
\item $\rho^{\lfloor p\rfloor} S(X)=X$;
\item for all $s\leq u\leq t$, 
\[
S(X)_{s,u}\star S(X)_{u,t}=S(X)_{s,t}\;;
\]
\item for all $n \geq 1$, there exists a control $\omega$ such that $|\pi^n S(X)_{s,t}|\leq\omega(s,t)^{n/p}$.
\end{enumerate}
\end{theorem}

One can show that the $S(X)_{s,t}\in \GG^*$ for all $s\leq t$.
The element $S(X)_{0,1}\in \mathcal{H}^*$ is called the \emph{signature} of $X$.

\subsection{GL algebra as a free algebra}

We recall the following result of Foissy~\cite[Sec.~8]{Foissy02} and Chapoton~\cite{Chapoton10}, which will play a central role in the sequel.
\begin{theorem}[\cite{Foissy02,Chapoton10}]\label{thm:gen}
There exists a subspace $B = \spn{\tau_1,\tau_2, \ldots }$ \label{B page ref} of $\BB$ such that $\HH$ is isomorphic as a Hopf algebra to the tensor Hopf algebra $T(B)$.
\end{theorem}

We note that the subspace $B$ (and evidently choice of basis $\tau_1,\tau_2,\ldots$) is not unique.
However, we can and will assume that we have a fixed set of basis elements $\tau_1,\tau_2,\ldots$ which are in $\TT$ and that $|\tau_i| \leq |\tau_j|$ if $i < j$ (cf.~\cite[p.~106]{Foissy02}).
\begin{remark}\label{rem:different B basis}
This assumption on $\tau_1,\tau_2,\ldots$ is only for convenience; all our arguments carry through if instead $\tau_i = \sum_{j=1}^{n_i}c_i^{(j)}\tau_i^{(j)}$, where every $\tau_i^{(j)}$ is an element of $\TT$ with $|\tau_i^{(1)}| = \ldots = |\tau_i^{(n_i)}|$.
\end{remark}
Due to the grading of $\HH$, it follows that every element $\sigma \in \HH^*$ can be written uniquely as
\begin{equation}\label{eq:HH* series}
\sigma = \sum_{m=0}^\infty \sum_{r_1,\ldots, r_m = 1}^\infty \lambda_{r_1,\ldots, r_m} \tau_{r_1}\star \ldots \star \tau_{r_m}
\end{equation}
(the term with $m=0$ corresponds to a linear multiple of the empty forest $\1$).
Note that $\sigma \in \HH$ if and only if all but finitely many of the $\lambda$'s are zero.

\section{Geometric \texorpdfstring{$\Pi$}{Pi}-rough paths}
\label{sec:Pi rough paths}

In order to describe how every branched rough path is canonically a geometric rough path, we find it most natural to work with a generalisation of geometric $p$-rough paths which allows one to measure different components with different regularity; the notion of a $\Pi$-rough path~\cite{Gyurko08,Gyurko16}, which is a generalisation of $(p,q)$-rough paths~\cite{LV06}, provides a convenient framework for this purpose.

\begin{remark}\label{rem:reg structures}
Having different regularities for different components of the ``driver'' is also naturally achieved in the theory of regularity structures.
In fact, the restriction of the algebraic formalism in~\cite{BHZ16} to the one-dimensional setting (and ignoring the presence of polynomials and derivatives on kernels) yields precisely a ``branched'' version of $\Pi$-rough paths.
The isomorphism between branched and geometric rough paths from Theorem~\ref{thm:branched are geometric} then carries over to this setting and reads: every branched $\Pi$-rough path is canonically a geometric $\tilde \Pi$-rough path, where $\tilde \Pi$ is a tuple of regularities which depends on $\Pi$.
In fact, we work with branched $p$-rough paths (i.e., every component of the driver has the same regularity) only for simplicity.
\end{remark}

Let $V$ be a finite dimensional normed vector space.
We represent different homogeneities of a rough path by decomposing the space $V$ into
\begin{equation}\label{eq:V decomp}
V=V^{1}\oplus\ldots\oplus V^{k}\;.
\end{equation}
We define the set of multi-indexes \label{mathcal{A}_k page ref}
\[
\mathcal{A}_{k}=\{(r_{1},\ldots,r_{m}):r_{i}\in\{1,\ldots,k\},\;m \geq 0\}\;.
\]
Due to the decomposition~\eqref{eq:V decomp}, we may write 
\[
T((V))=\prod_{(r_{1},\ldots,r_{m})\in\mathcal{A}_{k}}V^{\otimes R}\;,
\]
where for $R=(r_{1},\ldots,r_{m}) \in \mathcal{A}_{k}$ we denote $V^{\otimes R} = V^{r_1}\otimes\ldots\otimes V^{r_m}$. \label{V^{otimes R} page ref}
We let $\pi_{R} : T((V)) \to V^{\otimes R}$ \label{pi_R page ref} denote the corresponding projection.

Consider a scaling tuple $\Pi=(p_{1},\ldots,p_{k})$ with $p_1\geq \ldots \geq p_k \geq 1$.
We would like
to scale $V^{1}$ by $\frac{1}{p_{1}}$, $V^{2}$ by $\frac{1}{p_{2}}$,
etc.
For a multi-index $R=(r_{1},\ldots,r_{m}) \in \mathcal{A}_{k}$, we set $n_{j}(R)=|\{i:r_{i}=j\}|$ \label{n_j(R) page ref} for all $j=1,\ldots, k$, and define the scaling for $V^{\otimes R}$
by \label{deg_{Pi}(R) page ref}
\[
\mbox{deg}_{\Pi}(R) = \sum_{j=1}^{k}\frac{n_{j}(R)}{p_{j}}\;,
\]
which is the inhomogeneous analogue of the tensor degree.

For $s \geq 0$, define further \label{mathcal{A}^{Pi}_s page ref}
\[
\mathcal{A}_{s}^{\Pi}=\{R\in\mathcal{A}_{k}:\mbox{deg}_{\Pi}(R)\leq s\}.
\]
\begin{remark}
Note that the set $\mathcal{A}_{s}^{\Pi}$ is finite.
In the spirit of Remark~\ref{rem:reg structures}, this fact is directly related to the notion of subcriticality in regularity structures.
\end{remark}
Consider the ideal
\[
B_{s}^{\Pi}=\{v\in T((V)):\pi_{R}(v)=0\quad \forall R\in\mathcal{A}_{s}^{\Pi}\}\;.
\]
We define the truncated tensor
algebra at degree $s$ as the quotient
algebra \label{T^{(Pi,s)}(V) page ref}
\[
T^{(\Pi,s)}(V)=T((V))/B_{s}^{\Pi}\;.
\]

\begin{definition}
A $\Pi$-rough path is a continuous function $X:\triangle\rightarrow T^{(\Pi,1)}(V)$ such that
\begin{enumerate}[label=(\arabic*)]
\item for all $s\leq u\leq t$
\[
X_{s,u}\dot\otimes X_{u,t}=X_{s,t}\; ;
\]
\item \label{point:pvar Pi rps}
there exists a control $\omega$ such that for all $R\in \mathcal{A}^\Pi_1$ and $s \leq t$
\begin{equation*}
|\pi_{R}(X_{s,t})|\leq \omega(s,t)^{\deg_{\Pi}(R)}\;.
\end{equation*}
\end{enumerate}
\end{definition}

The usual extension theorem for $p$-rough paths holds for $\Pi$-rough paths, namely that a $\Pi$-rough path $X$ has a unique extension $S(X) : \triangle \to T((V))$ which preserves the multiplicative and regularity properties, see~\cite[Thm.~2.6]{Gyurko16}. 

Finally, consider the subgroup $G^{(\Pi,1)}(V)$ \label{G^{(Pi,1)}(V) page ref} of $T^{(\Pi,1)}(V)$ defined as the exponential of the Lie subalgebra of $T^{(\Pi,1)}(V)$ generated by $V$.
We say that a $\Pi$-rough path is \textit{weakly geometric} if it takes values in $G^{(\Pi,1)}(V)$.

\subsection{Paths in homogeneous groups}
We describe now how weakly geometric $\Pi$-rough paths can be treated as paths taking values in a homogeneous group.
The advantage of this viewpoint is that it gives rise to a homogeneous $p$-variation metric which is convenient to work with in practice (particularly for establishing interpolation estimates).
The results here will be helpful later in Section~\ref{subsec:unique sigs}.

We introduce on $T^{(\Pi,1)}(V)$ a family of algebra morphisms $(\delta_\lambda)_{\lambda > 0}$, called dilation operators, given for any $v \in V^{\otimes R}$ by $\delta_\lambda(v) = \lambda^{p_1 \deg_{\Pi}(R)}v$.
Note that the restriction of $(\delta_\lambda)_{\lambda > 0}$ to $G^{(\Pi,1)}(V)$ defines a homogeneous group in the sense of Folland--Stein~\cite{FollandStein82} (which is a Carnot group if and only if $p_1=\ldots =p_k=1$).

We equip $G^{(\Pi,1)}(V)$ with a sub-additive homogeneous norm and corresponding left-invariant metric $d$, see~\cite{HebischSikora90}.
Consequently, we define the $p$-variation norm $\|\cdot\|_{p\var}$ \label{Vert cdot Vert_{pvar} page ref} and (homogeneous) $p$-variation metric $d_{p\var}$ \label{d_{pvar} page ref} on the space of functions from $[0,1]$ to $G^{(\Pi,1)}(V)$ (see, e.g.,~\cite[Sec.~3.1]{Chevyrev18}).

For maps $X,Y : \triangle \to T^{(\Pi,1)}(V)$ and any tuple $\Pi'=(p_1',\ldots,p_k')$, we define the (inhomogeneous) $\Pi'$-variation distance \label{rho_{Pi-var} page ref}
\[
\rho_{\Pi'\var}(\bar X,\bar Y) = \max_{R \in \mathcal{A}^\Pi_1} \sup_{\PP} \Big(\sum_{[s,t] \in \PP} |\pi_R(\bar X_{s,t} - \bar Y_{s,t})|^{1/\deg_{\Pi'}(R)} \Big)^{\deg_{\Pi'}(R)}\;,
\]
where $\sup_{\PP}$ denotes the supremum over all finite partitions of $[0,1]$ and $\sum_{[s,t] \in \PP}$ denotes the sum over adjacent point in $\PP$.
Note that there is a one-to-one correspondence between multiplicative maps $X:\triangle \to G^{(\Pi,1)}(V)$ and paths $\tilde X : [0,1] \to G^{(\Pi,1)}(V)$ with $\tilde X_0 = \1$, given by $X_{s,t} = \tilde X_s^{-1}\dot\otimes \tilde X_t$.
An analogue of the ball-box estimate~\cite[Prop.~7.49]{FrizVictoir10} yields the following lemma.
\begin{lemma}\label{lem:LipschitzHolder}
Let $p\geq 1$.
Define $p'_j = \frac{p}{p_1}p_j$ and $\Pi' = (p'_1,\ldots,p'_k)$.
Then the identity map
\[
(C^{p\var}([0,1], G^{(\Pi,1)}(V)), d_{p\var}) \leftrightarrows (C^{p\var}([0,1], G^{(\Pi,1)}(V)), \rho_{\Pi'\var})
\]
is Lipschitz on bounded sets in the $\rightarrow$ direction, and H{\"o}lder continuous on bounded sets in the $\leftarrow$ direction.
\end{lemma}
As usual in rough paths theory, the use of this lemma comes from the fact that the It{\^o}-Lyons solution map for RDEs is locally Lipschitz for the metric $\rho_{\Pi\var}$, while interpolation estimates are most easily derived for the metric $d_{p\var}$.
For example, for any $0<p<p'$, it holds that (e.g.,~\cite[Lem.~3.3]{Chevyrev18}) 
\begin{equation}\label{eq:interpolation}
\sup_{t \in [0,1]}d(X^n_t, X_t) \to 0 \textnormal{ and } \sup_n\|X^n\|_{p\var} < \infty \Rightarrow d_{p'\var}(X^n,X) \to 0\;.
\end{equation}
A useful consequence is the following.
\begin{lemma}\label{lem:weakly geo is geo}
Let $X$ be a weakly geometric $\Pi$-rough path.
Let $p>p_1$ and define $\Pi'$ as in Lemma~\ref{lem:LipschitzHolder}.
Suppose that $p$ is sufficiently close to $p_1$ so that $T^{(\Pi',1)}(V) = T^{(\Pi,1)}(V)$.
Then $X$ is a geometric\footnote{A $\Pi$-rough path is called geometric if it is in the closure under $\rho_{\Pi\var}$ of the lifts of $V$-valued bounded variation paths.} $\Pi'$-rough path.
\end{lemma}

\begin{proof}
Observe that there exists $C>0$ such that if $\gamma : [0,1] \to G^{(\Pi,1)}(V)$ is the lift of a bounded variation path $\tilde \gamma : [0,1] \to V$ with $|\tilde \gamma|_{1\var} \leq 1$, then
\begin{equation}\label{eq:p1/pk 1 var bound}
\|\gamma\|_{(p_1/p_k)\var}^{p_1/p_k} \leq C |\tilde\gamma|_{1\var}\;.
\end{equation} 
By the Rashevsky-Chow theorem, for every $y \in G^{(\Pi,1)}(V)$ there exists a path $\gamma : [0,1] \to G^{(\Pi,1)}(V)$, which is the lift of a Lipschitz path in $V$, such that $\gamma_0 = \1$ and $\gamma_1 = y$.
By the Arzel{\`a}-Ascoli theorem and the lower semi-continuity of $\left\Vert \cdot \right\Vert_{(p_1/p_k)\var}$, there exists $\gamma$ (not necessarily unique) for which $\|\gamma\|_{(p_1/p_k)\var}$ is minimal amongst all such paths.
We call such a $\gamma$ a $(p_1/p_k)$-geodesic from $\1$ to $y$.
By~\eqref{eq:p1/pk 1 var bound}, we see that $\|\gamma\|_{(p_1/p_k)\var} \to 0$ as $y \to \1$.
It follows that $\|y\|' = \|\gamma\|_{(p_1/p_k)\var}$ defines a sub-additive homogeneous norm on $G^{(\Pi,1)}(V)$, and is therefore equivalent to $\|\cdot\|$.

Let $\PP_n$ be a sequence of partitions for which the mesh size $|\PP_n| \to 0$.
Let $X^n$ be a piecewise $(p_1/p_k)$-geodesic approximation to $X$ over $\PP_n$, i.e., for every $[s,t] \in \PP_n$, the path $[s,t] \to G^{(\Pi,1)}(V), u \mapsto X^n_{s,u}$, is a (reparametrisation of a) $(p_1/p_k)$-geodesic from $\1$ to $X_{s,t}$.
It clearly holds that $X^n\to X$ uniformly.
Furthermore, using the equivalence of $\|\cdot\|$ and $\|\cdot\|'$, one can apply~\cite[Lem.~A.5]{Chevyrev18} to show that $\|X^n\|_{p_1\var} \lesssim \|X\|_{p_1\var}$.
It follows by~\eqref{eq:interpolation} and Lemma~\ref{lem:LipschitzHolder} that $X$ is a geometric $\Pi'$-rough path.
\end{proof}

\section{Branched rough paths are geometric \texorpdfstring{$\Pi$}{Pi}-rough paths}
\label{sec:branched are geo}

We proceed to give an explicit isomorphism between branched $p$-rough paths and weakly geometric $\Pi$-rough paths.
The key remark throughout this section, which is a direct consequence of Theorem~\ref{thm:gen}, is that every element in $\HH^N$ can be uniquely written as
\begin{equation}\label{eq:expression in HH^N}
\sum_{R} \lambda_R \tau_{r_1}\star\ldots\star \tau_{r_m}\;,
\end{equation}
where $R=(r_1,\ldots, r_m)$ ranges over all multi-indexes for which $\sum_{j=1}^m|\tau_{r_j}| \leq N$.

\begin{definition}\label{def:Bk def}
Let $k \geq 1$ be the largest integer such that $|\tau_k| \leq p$.
Consider the decomposition into one-dimensional subspaces
\begin{equation}\label{eq:Bn decomp}
B_k = B^1\oplus\ldots\oplus B^k\;,
\end{equation}
\label{B_k page ref} where $B^j = \spn{\tau_j}$.
Define $p_j = p/|\tau_j|$ and the tuple
\[
\Pi = (p_1,p_2,\ldots, p_{k})\;.
\]
\end{definition}

We now specialise all the notation from Section~\ref{sec:Pi rough paths} by setting $B_k = V$ with the corresponding decomposition~\eqref{eq:Bn decomp}. 
Note that, by construction, every element $\sigma \in T^{(\Pi,1)}(B_k)$ can be written uniquely as
\[
\sigma = \sum_{(r_1,\ldots,r_m) \in \mathcal{A}^\Pi_1} \lambda_R \tau_{r_1}\dot\otimes\ldots \dot\otimes \tau_{r_m}\;,
\]
where $\mathcal{A}^\Pi_1$ is the set of all multi-indexes $R=(r_1,\ldots, r_m) \in \{1,\ldots, k\}^m$, $m \geq 0$, for which $\deg_\Pi(R) = \sum_{j=1}^m |\tau_{r_j}|/p \leq 1$.

\begin{lemma}\label{lem:HN Bn isom}
There is an algebra isomorphism $\Psi : \HH^{\floor p} \to T^{(\Pi,1)}(B_k)$ \label{Psi page ref} given, for all $R = (r_1,\ldots r_m) \in \mathcal{A}^\Pi_1$, by
\[
\Psi : \tau_{r_1}\star\ldots \star \tau_{r_m} \mapsto \tau_{r_1}\dot\otimes \ldots \dot\otimes \tau_{r_m}\;.
\]
Furthermore, there exists $C>0$ such that for all $\tau \in \HH^{\floor p}$ and $n=1,\ldots, \floor p$, it holds that
\begin{equation}\label{eq:two sided bound}
\sup_{\deg_{\Pi}(R) = n/p} C^{-1}|\pi_R\Psi(\tau)| \leq |\pi^n\tau| \leq \sup_{\deg_{\Pi}(R) = n/p} C|\pi_R\Psi(\tau)|\;.
\end{equation}
\end{lemma}

\begin{proof}
The existence of $\Psi$ is immediate from the definitions and Theorem~\ref{thm:gen}, while~\eqref{eq:two sided bound} follows from the equivalence of norms on the finite-dimensional space $\HH^{(n)}$.
\end{proof}

Recall that the space of maps from $\triangle$ into $\HH^{\floor p}$  and $T^{(\Pi,1)}(B_n)$ can be equipped respectively with the (inhomogeneous) $p$- and $\Pi$-variation metrics
\begin{align*}
\rho_{p\var}(X,Y) &= \max_{n=1,\ldots, \floor p} \sup_{\PP} \Big(\sum_{[s,t] \in \PP} |\pi^n(X_{s,t} - Y_{s,t})|^{p/n} \Big)^{n/p}\;,
\\
\rho_{\Pi\var}(\bar X,\bar Y) &= \max_{R \in \mathcal{A}^\Pi_1} \sup_{\PP} \Big(\sum_{[s,t] \in \PP} |\pi_R(\bar X_{s,t} - \bar Y_{s,t})|^{1/\deg_\Pi(R)} \Big)^{\deg_\Pi(R)}\;,
\end{align*}
where, as before, $\sup_{\PP}$ denotes the supremum over all finite partitions of $[0,1]$ and $\sum_{[s,t] \in \PP}$ denotes the sum over adjacent point in $\PP$.

\begin{theorem}\label{thm:branched are geometric}
Consider a map $X: \triangle \to \HH^{\floor p}$.
For $\Psi$ as in Lemma~\ref{lem:HN Bn isom}, let $\bar X := \Psi(X): \triangle \to T^{(\Pi,1)}(B_k)$. \label{bar{X} page ref}
Then
\begin{enumerate}[label=(\alph*)]
\item \label{point: main thm 1} $X$ is multiplicative if and only if $\bar X$ is multiplicative,
\item \label{point: main thm 2} $X$ takes values in $\GG^{\floor p}$ if and only if $\bar X$ takes values in $G^{(\Pi,1)}(B_k)$,
\item \label{point: main thm 3} there exists $C > 0$, independent of $X$, such that if $Y: \triangle \to \HH^{\floor p}$ is another map and $\bar Y := \Psi(Y)$, then
\begin{equation}\label{eq:p-Pi var bound}
C^{-1} \rho_{\Pi\var}(\bar X,\bar Y)
\leq \rho_{p\var}(X,Y)
\leq C \rho_{\Pi\var}(\bar X,\bar Y)\;;
\end{equation}
in particular, $X$ is a branched $p$-rough path if and only if $\bar X$ is a weakly geometric $\Pi$-rough path over $B_k$ associated with the decomposition~\eqref{eq:Bn decomp},

\item \label{point: main thm 4} suppose that $X$ is a branched $p$-rough path. 
Let $S(X) :\triangle \to \HH^*$ denote the extension of $X$ and $S(\bar X) : \triangle \to T((B_k))$ denote the extension of $\bar X$.
Then $S(X) = \imath S(\bar X)$, where $\imath : T((B_k)) \hookrightarrow P(B)$ is the natural inclusion map, and where we identify $\HH^*$ with $P(B)$ as in Section~\ref{subsec:top on HH}.
\end{enumerate}
\end{theorem}

\begin{proof}
Since $\Psi$ is an algebra isomorphism, it holds that $X$ is multiplicative if and only $\bar X$ is, which proves~\ref{point: main thm 1}.
Likewise, $\log(X)$ takes values in the Lie subalgebra of $\HH^{\floor p}$ generated by the trees $\tau_1,\ldots,\tau_k$ if and only if $\log(\bar X)$ takes values in the Lie subalgebra of $T^{(\Pi,1)}(B_k)$ generated by $\tau_1,\ldots, \tau_k$, which proves~\ref{point: main thm 2}.
Next, the bound~\eqref{eq:p-Pi var bound} is an immediate consequence of~\eqref{eq:two sided bound} and the definition of $\rho_{p\var}$ and $\rho_{\Pi\var}$, which proves~\ref{point: main thm 3}.

To prove~\ref{point: main thm 4}, consider $N \geq \floor{p}$ and $Z = \rho^N\imath S(\bar X) : \triangle \to \HH^N$.
It is immediate that $\rho^{\floor p}Z = X$, and, since $\imath$ and $\rho^N$ are algebra morphisms, that $Z$ is a multiplicative map.
Furthermore, using that $S(\bar X)$ has finite $\Pi$-variation~\cite[Thm.~2.6]{Gyurko16}, it readily follows from an analogous bound to~\eqref{eq:two sided bound} that $Z$ has finite $p$-variation.
By uniqueness of the branched rough path lift, it follows that $Z = \rho^N S(X)$, and thus $S(X) = \imath S(\bar X)$ as desired.
\end{proof}

\begin{remark}
The reader may wonder how canonical our interpretation of $X$ as a geometric $\Pi$-rough path is, given that the space $B$ and the decomposition~\eqref{eq:Bn decomp} depend on the choice of basis $\tau_1,\tau_2,\ldots$.
It is easy to see, however, that a different choice of $\tau_1,\tau_2,\ldots$ will lead to canonically isomorphic objects (provided $\tau_i$ are chosen in accordance with Remark~\ref{rem:different B basis}).
\end{remark}

\begin{remark}
The isomorphism between geometric and non-geometric rough paths shown in~\cite{LV06} is precisely the case $p\in (2,3)$ of Theorem~\ref{thm:branched are geometric}.
\end{remark}

\begin{remark}
In~\cite{HairerKelly15}, for any branched $p$-rough path, the authors employ the Lyons--Victoir extension theorem~\cite{LyonsVictoir07} to construct a geometric $p$-rough path taking values in the tensor algebra over $\BB^{\floor p} = \spn{\tau \in \TT : |\tau| \leq {\floor p} }$.
In contrast, the isomorphism in Theorem~\ref{thm:branched are geometric} does not ``extend'' the branched rough path in any way and instead treats its target space as a different algebraic structure.
The explicit nature of this isomorphism will be particularly important in our study of the signature in Section~\ref{sec:sigs of branched}.
\end{remark}

We now present two consequences of Theorem~\ref{thm:branched are geometric}.
For the remainder of this section, suppose that $X : \triangle \to \GG^{\floor p}$ is a branched $p$-rough path and that $\bar X = \Psi(X) : \triangle \to T^{(B,1)}(B_k)$.
As before, denote by $S(X) : \triangle \to \HH^*$ the extension of $X$.

First, it follows from Theorem~\ref{thm:branched are geometric} that the level-$N$ lift of a branched $p$-rough path is the solution of a linear differential equation driven by a geometric $\Pi$-rough path.

\begin{corollary}
Let $N \geq \floor{p}$ and let $Y = \rho^N S(X)_{0,\cdot} :[0,1]\to\GG^N$.
Then $Y$ is the solution of the linear RDE
\[
\mathrm{d}Y = f(Y) \mathrm{d} \bar X\;,
\]
where $f = (f_1,\ldots, f_k)$ are the (left-invariant) vector fields on $\HH^N$ given by right-multiplication by $(\tau_1,\ldots,\tau_k)$ respectively.
\end{corollary}

The interest in the above corollary stems from the fact that in general one is not able to describe $Y$ as the solution of a linear RDE driven by the original branched rough path $X$, cf.~\cite[Rem.~34]{BCFP17}.

Second, we show an It{\^o}-type formula that any RDE driven by $X$ coincides in a natural way with an RDE driven by $\bar X$.
For the remainder of this section, consider bounded smooth vector fields $f=(f_1,\ldots, f_d)$ on $\R^e$ with bounded derivatives of all orders.\footnote{The regularity assumptions on $f$ can be significantly weakened, see the sharp version of the universal limit theorem for geometric $\Pi$-rough paths~\cite[Thm.~4.3]{Gyurko16};
we restrict to smooth vector fields only for simplicity.}

Recall the pre-Lie product $\curvearrowright: \BB\times\BB \to \BB$ defined by $\tau\curvearrowright\sigma = \pi_{\BB}(\tau\star \sigma)$ \label{curvearrowright page ref}, where $\pi_{\BB} : \HH \to \BB$ is the projection onto $\BB$.
Explicitly, $\tau \curvearrowright \sigma = \sum_{\bar\tau} n(\tau,\sigma,\bar\tau)\bar\tau$,
where the sum is over all trees $\bar\tau \in \TT$ and $n(\tau,\sigma,\bar\tau)$ is the number of single admissible cuts of $\bar\tau$ for which the branch is $\tau$ and the trunk is $\sigma$.
Recall also that the space of vector fields $C^\infty(\R^e,\R^e)$ can be equipped with a pre-Lie product defined by $f\triangleleft g = \sum_{i=1}^e f^i\partial_i g$.
By a result of Chapoton--Livernet~\cite{Chapoton01}, we can identify $\BB$ with the free pre-Lie algebra over $\R^d$, and thus there exists a unique pre-Lie algebra morphism $\BB \to C^\infty(\R^e,\R^e)$, $\tau \mapsto f_\tau$, for which $f_{\bullet_i} = f_i$ for all $i = 1, \ldots, d$.

\begin{remark}\label{rem:[tree] notation}
Every tree $\tau \in \TT$ can be written as $\tau = [\sigma_1\ldots\sigma_n]_{i}$ (which is unique up to permutation of the $\sigma_j$), by which we mean that $\tau$ is formed by attaching the trees $\sigma_1,\ldots,\sigma_n \in \TT$ to a root with label $i \in \{1,\ldots, d\}$ (if $n=0$, we have $\tau = \bullet_i$).
For $\tau = [\sigma_1\ldots\sigma_n]_{i} \in \TT$, the vector field $f_\tau$ admits the inductive form
\begin{equation}\label{eq:ftau}
f_\tau = c_{\tau}(D^n f_i)(f_{\sigma_1},\ldots,f_{\sigma_n})\;,
\end{equation}
where $c_\tau$ is a combinatorial factor expressible in terms of the symmetries of $\tau$.
\end{remark}

\begin{proposition}\label{prop:Ito-Strat correction}
Define the vector fields $\bar f = (\bar f_1,\ldots, \bar f_k) = (f_{\tau_1},\ldots, f_{\tau_k})$ on $\R^e$.
Then the unique solutions to the (branched) RDE $\mathrm{d}Y = f(Y)\mathrm{d}X$ and the (geometric) RDE $\mathrm{d}\bar Y = \bar f(\bar Y) \mathrm{d}\bar X$ coincide.
\end{proposition}

\begin{proof}
Recall that $Y$ and $\bar Y$ are characterised by the Euler estimates
\[
Y_{s,t} = \sum_{|\tau| \leq \floor{p}} f_{\pi_B X_{s,t}}(Y_s) + o(\omega(s,t))
\]
(where we treat $X_{s,t}$ as an element of $\HH$ by the embedding $\HH^{\floor p} \hookrightarrow \HH$) and
\[
\bar Y_{s,t} = \sum_{(r_1,\ldots,r_m) \in \A^\Pi_1}\bar f_{r_1}\ldots\bar f_{r_m} I (\bar Y_s) \gen{\bar X_{s,t}, \tau_{r_1}\dot\otimes\ldots\dot\otimes \tau_{r_m}} + o(\omega(s,t))\;,
\]
where $\omega$ is a control on the $p$-variation of $X$ and $\bar X$ (for the former, see~\cite[Prop.~3.8]{HairerKelly15};\footnote{Note though that the factors $c_\tau$ in~\eqref{eq:ftau} are missing from the definition of $f_\tau$ in~\cite{HairerKelly15}} for the latter, see~\cite[Cor.~10.15]{FrizVictoir10}).
To conclude that $Y$ and $\bar Y$ coincide, it remains to observe that
\[
\sum_{(r_1,\ldots,r_m) \in \A^\Pi_1} \bar f_{r_1}\ldots\bar f_{r_m} I (y) \gen{\bar X_{s,t}, \tau_{r_1}\dot\otimes\ldots\dot\otimes \tau_{r_m}}
= \bar f_{X_{s,t}} I(y)
= f_{\pi_\BB X_{s,t}}(y)\;,
\]
where $\bar f_{\tau}$ denotes the image of $\tau \in \HH$ under the unique algebra morphism $\HH \to \OO(\R^e)$ which maps $\tau_{r} \mapsto f_{\tau_r}$ for $r=1,2,\ldots$, where $\OO(\R^e)$ is the algebra of differential operators on $\R^e$ (this algebra morphism exists due to Theorem~\ref{thm:gen}).
\end{proof}

\section{Signatures of branched rough paths}
\label{sec:sigs of branched}

\subsection{Uniqueness of signatures}
\label{subsec:unique sigs}

We now apply the identification of branched $p$-rough paths and geometric $\Pi$-rough paths to prove the following characterisation of branched rough paths with trivial signature.

For a topological space $\SSS$, recall that a continuous path $X : [0,1] \to \SSS$ is called {\it tree-like} if there exists an $\R$-tree $\Tfrak$, a continuous function $\phi : [0,1] \to \Tfrak$, and a map $\psi : \Tfrak \to \SSS$ such that $\phi(0)=\phi(1)$ and $X = \psi\circ\phi$.

\begin{theorem}\label{thm:unique}
Let $X : \triangle \mapsto \GG^{\floor p}$ be a branched $p$-rough path.
Then $S(X)_{0,1} = \1$ if and only if $X_{0,\cdot}$ is tree-like. 
\end{theorem}

We will first prove that a tree-like branched rough path has trivial signature, which, by Theorem~\ref{thm:branched are geometric}, is equivalent to showing that a weakly geometric tree-like $\Pi$-rough path has trivial signature. The proof is effectively identical to that of geometric rough paths case~\cite[Thm.~1.1]{Boedihardjo16}, but we find it necessary to emphasise several details.

\begin{lemma}
\label{lem:Inverse}
Let notation be as in Section~\ref{sec:Pi rough paths}.
Let $X : \triangle \to T^{(\Pi,1)}(V)$ be a weakly geometric $\Pi$-rough path, and define 
\[
(\overleftarrow{X})_{s,t}=X_{1-t,1-s}^{-1}\;.
\]
Then for all $s\leq t$,
\[
S(\overleftarrow{X})_{1-t,1-s}\dot\otimes S(X)_{s,t}=\1\;.
\]
\end{lemma}

\begin{proof}
The claim is clearly true if $X$ has bounded variation (as a path in $V$) and the conclusion follows by density and Lemma~\ref{lem:weakly geo is geo}.
\end{proof}

\begin{remark}
While we state Lemma~\ref{lem:Inverse} only for weakly geometric $\Pi$-rough paths, a direct (albeit more involved) argument shows that the same result holds true for any $\Pi$-rough path (not necessarily weakly geometric) and Banach space $V$.
\end{remark}

\begin{proposition}\label{prop:treelike implies trivial}
Let $X$ be a weakly geometric $\Pi$-rough path for which $X_{0,\cdot}$ is tree-like. Then $S(X)_{0,1}=\1$.
\end{proposition}

\begin{proof}
The proof in~\cite[Sec.~3]{Boedihardjo16} for weakly geometric $p$-rough paths carries over to our present setting mutatis mutandis.
Indeed,
\begin{itemize}
\item the ``central case''~\cite[Lem.~3.1]{Boedihardjo16} follows in the identical way by applying Lemma~\ref{lem:Inverse}, and
\item the proof in~\cite[Sec.~3.3]{Boedihardjo16} follows in the identical way once we use Lemma~\ref{lem:LipschitzHolder}, the interpolation result~\eqref{eq:interpolation}, and that the signature map is a continuous function in the metric $\rho_{\Pi'\var}$~\cite[Lem.~2.1.2]{Gyurko08}.
\end{itemize}
\end{proof}

\begin{proof}[Proof of Theorem~\ref{thm:unique}]
Denote as before $\bar X = \Psi(X)$.
Note that $\bar X$ is tree-like if and only if $X$ is.
Furthermore, by part~\ref{point: main thm 3} of Theorem~\ref{thm:branched are geometric}, $\bar X$ is a weakly geometric $\Pi$-rough path and, by part~\ref{point: main thm 4} of Theorem~\ref{thm:branched are geometric}, it holds that $S(X)_{0,1}=\1$ if and only if $S(\bar X)_{0,1}=\1$.
The ``if'' direction now follows from Proposition~\ref{prop:treelike implies trivial}, while the ``only if'' direction follows from the main result of~\cite{Boedihardjo16} (one simply notes that every weakly geometric $\Pi$-rough path lifts canonically to a weakly geometric $p_1$-rough path over $B_k$).
\end{proof}

\subsection{Fourier transform and moment problem}
\label{subsec:char funcs}

We now discuss the Fourier transform (or characteristic function) and moment problem for signatures of branched rough paths.
The results here employ Theorem~\ref{thm:gen} to identify $\HH$ with the tensor alegbra $T(B)$, which allows us to extend the methods of~\cite{ChevyrevLyons16}.

\subsubsection{Universal locally \texorpdfstring{$m$}{m}-convex algebra over $B$}\label{subsec:loc m convex alg}

We begin by constructing a certain universal topological algebra over $B$.

\begin{remark}\label{rem:c00}
Throughout this subsection, we use no special structure of $B$ and note that it may be replaced by $c_{00}$, the vector space of $\R$-valued sequences which are eventually zero.
\end{remark}

We equip $B$ with the product topology given by the sequence of semi-norms $(\gamma_k)_{k \geq 1}$ \label{gamma_k page ref}
\[
\gamma_{k}\Big(\sum_{r} \lambda_r \tau_r\Big) = \sum_{r=1}^k |\lambda_r|\;.
\] 
Let $E_a(B)$ \label{E_a(B) page ref} denote the topological algebra formed by equipping the tensor algebra $T(B)$ with the corresponding universal locally $m$-convex topology, see~\cite[Sec.~2]{ChevyrevLyons16}.
Explicitly, a fundamental family of sub-multiplicative semi-norms on $E_a(B)$ is given by $\{\exp(K\gamma_k)\}_{k \geq 1, K > 0}$, where
\[
\exp(K\gamma_k)= \sum_{n =0}^\infty K^n\gamma_k^{\otimes n}\;,
\]
and
\[
\gamma_k^{\otimes n} \Big(\sum_{m=0}^\infty \sum_{r_1,\ldots, r_m = 1}^\infty \lambda_{r_1,\ldots, r_m} \tau_{r_1}\dot\otimes\ldots \dot\otimes\tau_{r_m} \Big) = \sum_{r_1,\ldots, r_n = 1}^k |\lambda_{r_1,\ldots, r_n}|
\]
(as usual, the term with $m=0$ on the LHS corresponds to a linear multiple of $\1 \in B^{\otimes 0}$).
Note that all sums above are finite.

For $m \geq 1$, the (complete) locally convex space $B^{\hat\otimes m}$ can be identified with the space of formal series
\begin{equation}\label{eq:element of B otimes k}
\sum_{r_1,\ldots, r_m = 1}^\infty \lambda_{r_1,\ldots,r_m} \tau_{r_1}\dot\otimes\ldots \dot\otimes\tau_{r_m}\;,
\end{equation}
and so $P(B) = \prod_{m \geq 0} B^{\hat \otimes m}$ can be identified with the space of formal series
\[
\sum_{m =0}^\infty \sum_{r_1,\ldots, r_m = 1}^\infty \lambda_{r_1,\ldots,r_m} \tau_{r_1}\dot\otimes\ldots \dot\otimes\tau_{r_m}\;.
\]
Let $E(B)$ \label{E(B) page ref} denote the completion of $E_a(B)$. The following lemma is immediate from the above discussion (and is a special case of~\cite[Cor.~2.5]{ChevyrevLyons16}).

\begin{lemma}\label{lem:completion}
The space $E(B)$ can be identified with the subspace of $P(B)$ consisting of series
\[
\sigma = \sum_{m = 0}^\infty \sum_{r_1,\ldots, r_m = 1}^\infty \lambda_{r_1,\ldots, r_m} \tau_{r_1}\dot\otimes \ldots \dot\otimes\tau_{r_m}
\]
such that for every $k \geq 1$ and $K > 0$
\[
\exp(K \gamma_k)(\sigma) = \sum_{m=0}^\infty \sum_{r_1,\ldots, r_m = 1}^k K^m |\lambda_{r_1,\ldots, r_m}| < \infty\;.
\]
\end{lemma}

We note that $E(B)$ is metrizable and separable since $B$ is (see~\cite[Cor.~2.4]{ChevyrevLyons16}).
Moreover, recall from~\cite[Sec.~3]{ChevyrevLyons16} that the coproduct $\Delta : E_a(B) \to E_a(B) \otimes E_a(B)$, defined as usual by $\Delta(\tau) = \tau\otimes \1 + \1\otimes \tau$ for all $\tau \in B$ and extended uniquely as an algebra morphism, is continuous, and so extends to the completions $\Delta : E(B) \to E(B) \hat \otimes E(B)$.
\begin{definition}\label{def:G}
Let $G = \{g \in E(B) : \Delta g = g\otimes g, g\neq 0\}$ \label{G page ref} denote the set of group-like elements of $E(B)$.
\end{definition}
Note that $G$ is a Polish space and a topological group.

\subsubsection{Topology on \texorpdfstring{$\HH$}{H}}\label{subsec:top on HH}

By Theorem~\ref{thm:gen}, there is a Hopf algebra isomorphism $\HH \simeq E_a(B)$, and we henceforth equip $\HH$ with the locally $m$-convex topology induced by this isomorphism.
We let $\hat \HH$ \label{hat HH page ref} denote the completion of $\HH$ and note that the previous isomorphism extends to $\hat\HH \simeq E(B)$ as locally $m$-convex algebras.
We continue to use $\exp(K\gamma_k)$ for the semi-norms on $\HH$ induced by this isomorphism.
That is, a fundamental family of sub-multiplicative semi-norms on $\HH$ is given by $\{\exp(K\gamma_k)\}_{k \geq 1, K > 0}$, where
\begin{equation}\label{eq:exp K gamma def}
\exp(K\gamma_k)= \sum_{n =0}^\infty K^n\gamma_k^{\otimes n}\;,
\end{equation}
and
\[
\gamma_k^{\otimes n} \Big(\sum_{m=0}^\infty \sum_{r_1,\ldots, r_m = 1}^\infty \lambda_{r_1,\ldots, r_m} \tau_{r_1}\star \ldots \star \tau_{r_m} \Big)
= \sum_{r_1,\ldots, r_n = 1}^k |\lambda_{r_1,\ldots, r_n}|\;.
\]
Note that the above sums are all finite.
Remark that while the semi-norms $\exp(K\gamma_k)$ on $\HH$ depend on the choice of basis of $B$, the locally $m$-convex topology on $\HH$ does not.

Since every element in $\HH^*$ admits a unique representation as~\eqref{eq:HH* series}, we can identify $\HH^*$ with $P(B) = \prod_{m\geq 0} B^{\hat\otimes m}$.
In particular, by Lemma~\ref{lem:completion}, we can identify $\hat \HH$ with a subspace of $\HH^*$ consisting of formal series with a suitable decay condition.

Recall the set of group-like elements $\GG^* = \{g \in \HH^* : \delta g = g\otimes g, g\neq 0 \}$.

\begin{definition}\label{def:GGk}
Let $\GG = \GG^* \cap \hat \HH$ \label{GG page ref} denote the set of group-like elements in $\hat\HH$.
Furthermore, for $k \geq 1$, let $\GG_k = \GG \cap P(B_k)$ \label{GG_k page ref} denote the group-like elements in $\hat \HH$ generated by $\tau_1,\ldots,\tau_k$, where we canonically identify $P(B_k)$ with a subalgebra of $\HH^*$.
\end{definition}

Note that $\GG$ is precisely the image of $G$ under the isomorphism $E(B) \simeq \hat\HH$.
Equivalently, $\GG$ is the subset of all $g \in \GG^*$ for which $\exp(K\gamma_n)(g) < \infty$ for all $n \geq 1$ and $K > 0$.

Let $k$ be as in Definition~\ref{def:Bk def}.
Then $\GG_k$ consists of all $g \in \GG$ whose unique series representation as~\eqref{eq:HH* series} contains no terms $\tau$ for which $|\tau| > \floor p$,
i.e.,
\[
g = \sum_{m = 0}^\infty \sum_{r_1,\ldots, r_m=1}^k \lambda_{r_1,\ldots, r_m} \tau_{r_1}\star\ldots\star \tau_{r_m}\;.
\]
An immediate consequence of Theorem~\ref{thm:branched are geometric} part~\ref{point: main thm 4} along with the factorial decay of geometric $\Pi$-rough paths~\cite[Thm.~2.6]{Gyurko16}, is that the signature of a branched $p$-rough path takes values in $\GG_k$.
More precisely, we have the following corollary.

\begin{corollary}\label{cor:sig in subalgebra}
Let $p \geq 1$ and $X$ a branched $p$-rough path.
For every $(s,t) \in \triangle$, it holds that $S(X)_{s,t}$ can be uniquely written as
\[
S(X)_{s,t} = \sum_{m =0 }^\infty \sum_{r_1,\ldots, r_m=1}^k \lambda^{s,t}_{r_1,\ldots, r_m} \tau_{r_1}\star\ldots\star \tau_{r_m}\;,
\]
where, for all $K > 0$,
\[
\sum_{m = 0}^\infty \sum_{r_1,\ldots, r_m=1}^k K^m|\lambda_{r_1,\ldots, r_m}^{s,t}| < \infty\;.
\]
\end{corollary}

\begin{remark}
We note that it is difficult to use Theorem~\ref{thm:branched are geometric} to obtain information about the individual quantities $\gen{S(X)_{s,t},\tau}$ for $\tau \in \FF$.
Nonetheless, using an independent method, it was shown in~\cite[Thm.~4]{Boedihardjo15} that $|\gen{S(X)_{0,1},\tau}| \lesssim c^{|\tau|}\tau!^{-1/p}$ uniformly over $\tau \in \FF$, where $\tau!$ denotes the tree factorial.
\end{remark}

\subsubsection{Non-commutative Fourier transform}\label{subsubsec:Fourier}
For the remainder of the section, we identify $\HH$ with $E_a(B)$ and $\hat\HH$ with $E(B)$ as topological algebras;
we also identify $\GG$ with $G$ as topological (Polish) groups.

Our main motivation for the construction of $\hat\HH$ as the complete universal locally $m$-convex algebra over $B$ is that one can readily characterise the continuous representations of $\hat\HH$.
Indeed, for any Banach algebra $A$, there is a natural bijection between continuous linear maps $M : B \to A$ and continuous algebra morphisms $M : \hat\HH \to A$.
Furthermore, for any normed space $V$, the set of continuous linear maps $M : B \to V$ has a straightforward characterisation: a linear map $M : B \to V$ is continuous if and only if there exists $N\geq 1$ such that $M(\tau_j) = 0$ for all $j \geq N$. 

For a complex finite-dimensional Hilbert space $H$, let $\LLL(H)$ denote the algebra of linear operators on $H$, and $\uu(H)$ \label{uu(H) page ref} the Lie subalgebra of $\LLL(H)$ consisting of anti-Hermitian operators on $H$.

\begin{definition}
Let $\A$ \label{mathcal{A} page ref} denote the family of all continuous finite-dimensional algebra representations $M : \hat\HH \to \LLL(H_M)$ which arise from extensions of continuous linear maps $M : B \to \uu(H_M)$, where $H_M$ ranges over all finite-dimensional complex Hilbert space.
We define the corresponding set of matrix coefficients by \label{mathcal{C} page ref}
\[
\CC = \{ \sigma \mapsto \gen{M(\sigma)u,v}_{H_M} : M \in \A, u,v \in H_M\}\;.
\]
\end{definition}

Note that $\CC$ consists of $\C$-valued continuous linear functionals on $\hat\HH$.
It readily follows that the restriction of any $M \in \A$ to $\GG$ is a continuous group morphism into the compact group of unitary operators on $H_M$ (see~\cite[Sec.~4]{ChevyrevLyons16}).
Moreover, by considering adjoint and tensor product representations, one can easily show that $\restr{\CC}{\GG}$ is closed under multiplication and conjugation, and is therefore a $*$-subalgebra of $C_b(\GG,\C)$.
To summarise, we have the following lemma.
\begin{lemma}
The set $\restr{\CC}{\GG}$ is a subspace of $C_b(\GG,\C)$ which is closed under multiplication and conjugation.
\end{lemma}
In order to apply a Stone--Weierstrass argument, the final and deeper point left to observe is that $\CC$ separates the points of $\hat\HH$.
Remark that if $B$ were finite-dimensional, this would follow directly from~\cite[Thm.~4.8]{ChevyrevLyons16}.
Although $B$ is infinite dimensional,
it holds that every element in $B^{\hat\otimes m}$ admits the form~\eqref{eq:element of B otimes k}.
We can thus apply the result of Giambruno--Valenti~\cite[Thm.~6]{Giambruno95} on polynomial identities of symplectic Lie algebras to conclude that, for every non-zero $\tau \in B^{\hat\otimes m}$, there exists $M \in \A$ such that $M(\tau) \neq 0$.
It follows from the identification of $\hat\HH$ with a space of formal series from Lemma~\ref{lem:completion}
that for every $\tau \in \hat \HH$,
\[
M(\tau) = 0\;,\; \; \forall M \in \A \;\;\; \Leftrightarrow \;\;\; \tau = 0
\]
(cf. proof of~\cite[Thm.~4.8]{ChevyrevLyons16}).
Using a Stone--Weierstrass argument, we obtain the following generalisation of~\cite[Cor.~4.12]{ChevyrevLyons16}.

\begin{theorem}\label{thm:charFunc}
Let $\mu, \nu$ be two Borel probability measures on $\GG$.
It holds that $\mu = \nu$ if and only if $\mu(f) = \nu(f)$ for all $f \in \CC$, or equivalently, $\mu(M) = \nu(M)$ for all $M \in \A$.
\end{theorem}

Theorem~\ref{thm:charFunc} naturally suggests the following definition. 

\begin{definition}[Fourier transform]
The abstract non-commutative Fourier transform of a probability measure $\mu$ on $\GG$ is the map $\hat\mu : \CC \to \C$, $\hat\mu : f \mapsto \mu(f)$, or equivalently, the map $\hat\mu : M \mapsto \mu(M)$, where $M \in \A$.
\end{definition}

\subsubsection{Moment problem}
With the Fourier transform in hand, we are ready to address the moment-problem.
Let $\X$ be a $\GG$-valued random variable for which $\gen{\X,\sigma}$ is integrable for every $\sigma \in \FF$ (e.g., $\X = S(X)_{0,1}$ for a random branched $p$-rough path with controlled moments).

\begin{definition}
We call the element of $\HH^*$
\begin{equation}\label{eq:expsig def}
\ExpSig(\X) = \sum_{\sigma \in \FF} \EEE{\gen{\X,\sigma}} \sigma 
\end{equation}
the expected signature of $\X$.
\end{definition}

Combining Theorem~\ref{thm:charFunc} with~\cite[Prop.~3.2]{ChevyrevLyons16}, we arrive at the following solution to the moment problem.

\begin{proposition}[Moment problem]\label{prop:moment prob}
Suppose $\X$ is a $\GG$-valued random variable such that $\ExpSig(\X)$ exists and lies in $\hat \HH$, i.e.,
\begin{equation}\label{eq:bound on ExpSig}
\exp(K\gamma_k)\big(\ExpSig(\X)\big) < \infty\;, \quad \forall K > 0\;,\;k \geq 1\;,
\end{equation}
where $\exp(K\gamma_k)$ is given by~\eqref{eq:exp K gamma def}.
If $\Y$ is another $\GG$-valued random variable such that $\ExpSig(\X) = \ExpSig(\Y)$, then $\X$ and $\Y$ and equal in law.
\end{proposition}

\begin{remark}\label{rem:esig}
The reader may wonder the extent to which it is possible to control the quantities $\exp(K\gamma_k)\big(\ExpSig(\X)\big)$ given that it requires us to derive the form
\begin{equation}\label{eq:expsig alt form}
\ExpSig(\X) = \sum_{m=0}^\infty\sum_{r_1,\ldots,r_m}^\infty \lambda_{r_1,\ldots,r_m}\tau_{r_1}\star\ldots\star\tau_{r_m}\;,
\end{equation}
and there does not seem to be an easy way to determine the values $\lambda_{r_1,\ldots,r_m}$ from the expression~\eqref{eq:expsig def}.

In the case $\X=S(X)_{0,1}$ for a branched $p$-rough path $X$, however, we point out that the form~\eqref{eq:expsig alt form} can in fact arise more naturally than~\eqref{eq:expsig def} once we identify $X$ with a geometric rough path.
Furthermore, checking the bound~\eqref{eq:bound on ExpSig} can also become simpler due to the fact that
\begin{enumerate}[label=(\alph*)]
\item by Corollary~\ref{cor:sig in subalgebra}, the expression~\eqref{eq:expsig alt form} has no terms $\tau_r$ with $|\tau_r| > p$, and thus it suffices to check~\eqref{eq:bound on ExpSig} only for $k$ as in Definition~\ref{def:Bk def},
\item a number of methods are available to bound the expected signature of geometric rough paths, one of the most applicable being better-than-exponential tails on the local $p$-variation~\cite{CassLittererLyons13}, see~\cite[Cor.~6.5]{ChevyrevLyons16}.
\end{enumerate}
\end{remark}

\section{Example: It{\^o}-lift of a semi-martingale}
\label{subsec:example}
Suppose $p \in (2,3)$ and $(X^i)_{i=1}^d : [0,1] \to \R^d$ is a semi-martingale.
Let $X : \triangle \to \GG^2$ be the corresponding It{\^o} branched $p$-rough path defined by
\[
X_{s,t}
= \1 + \sum_{i=1}^d \bullet_i X^i_{s,t} + \sum_{1 \leq i \leq j \leq d} \bullet_i\bullet_j X^i_{s,t}X^j_{s,t}  + \sum_{i,j=1}^d [\bullet_i]_j \int_{s}^t X^i_{s,u} \mathrm{d}X^j_u \;,
\]
where $[\bullet_i]_j$ is defined in Remark~\ref{rem:[tree] notation} and the integral is defined in the sense of It{\^o}. Recall the subspace of trees $B_k$ from Definition~\ref{def:Bk def}.
Using the basic identity\footnote{In practice, one computes $\tau\star\sigma$ by the definition of $\star$ as the dual of the Connes-Kreimer coproduct which in turn is defined in terms of admissible cuts.}
\[
\bullet_i \star \bullet_j =
\begin{cases}
\bullet_i\bullet_j + [\bullet_i]_j &\textnormal{ if } i \neq j\;,
\\
2\bullet_i\bullet_i + [\bullet_i]_i &\textnormal{ if } i = j\;,
\end{cases}
\]
it is easy to check in this case that $k = \frac{d(d+3)}{2}$ and a suitable basis for $B_k$ is
\[
\{\bullet_i :  1\leq i \leq d\} \cup \{[\bullet_i]_j : 1 \leq i \leq j \leq d\}\;.
\]
Furthermore, using the identities
\begin{align*}
&\int_{s}^t X^i_{s,u}\mathrm{d}X^j_u = \int_{s}^t X^i_{s,u} \circ \mathrm{d} X^j_u - \frac{1}{2}[X^i,X^j]_{s,t}\;,
\\
&X^i_{s,t}X^j_{s,t} = \int_{s}^t X^i_{s,u} \circ \mathrm{d}X^j_u + \int_{s}^t X^j_{s,u} \circ \mathrm{d}X^i_u \;,
\end{align*}
where $\circ \mathrm{d}$ denotes the Stratonovich differential, we see that the unique way to write $X_{s,t}$ in the form~\eqref{eq:expression in HH^N} is
\begin{equation}\label{eq:Ito-Strat_conversion}
\begin{split}
X_{s,t}
&=
\1 + \sum_{i=1}^d X^i_{s,t}\bullet_i + \sum_{i,j=1}^d \int_s^t X^i_{s,u}\circ \mathrm{d}X^j_u \bullet_i \star \bullet_j
\\
&\quad +\frac{1}{2}\sum_{1 \leq i < j \leq d} [X^i,X^j]_{s,t}[\bullet_i,\bullet_j]
- \sum_{1 \leq i < j \leq d} [X^i,X^j]_{s,t}[\bullet_i]_j
\\
&\quad - \frac{1}{2}\sum_{i=1}^d [X^i,X^i]_{s,t}[\bullet_i]_i\;,
\end{split}
\end{equation}
where $[\bullet_i,\bullet_j] = \bullet_i\star\bullet_j - \bullet_j\star\bullet_i$ is the usual Lie bracket in $\HH^2$.

It follows that the corresponding geometric $\Pi$-rough path (or simply $p$-rough path if we ignore the refined $\Pi$-variation) over $B_k$ is given, in the first $d$ components, by the canonical (Stratonovich) geometric lift of $(X^i)_{i=1}^d$ with a bounded variation drift added to the L{\'e}vy area, and, in the remaining $\frac{d(d+1)}{2}$ components, by another bounded variation drift. 

Let us now specialise to the case that $(X^i)_{i=1}^d$ is a Brownian motion with zero mean and covariance $[X^i,X^j]_{s,t}=\Sigma^{i,j}(t-s)$.
We can treat $X$ as a $G^2(B_k)$-valued continuous L{\'e}vy process and apply~\cite[Thm.~53]{FrizShekhar17} to obtain the formula
\begin{multline*}
\ExpSig(S(X)_{0,t}) = \exp_{\star}\Big[ t \Big( 
\frac{1}{2}\sum_{i,j=1}^d \Sigma^{i,j} \bullet_i\star\bullet_i
+ \frac{1}{2}\sum_{1\leq i < j \leq d} \Sigma^{i,j}[\bullet_i,\bullet_j]
\\
- \sum_{1\leq i < j \leq d}\Sigma^{i,j}[\bullet_i]_j  
-\frac{1}{2}\sum_{i=1}^d \Sigma^{i,i}[\bullet_i]_i
\Big) \Big]\;,
\end{multline*}
from which it is manifest that $\ExpSig(S(X)_{0,t})$ satisfies~\eqref{eq:bound on ExpSig}.
We conclude by Proposition~\ref{prop:moment prob} that $S(X)_{0,t}$ is the unique $\GG$-valued random variable with the above expected signature.

\begin{remark}
As pointed out in Remark~\ref{rem:esig}, an alternative method to check the bound~\eqref{eq:bound on ExpSig} is to note that we have sufficient bounds on the local $p$-variation of $X$ (as a random geometric rough path).
The advantage of this method is that it readily generalises to stochastic processes for which an explicit form of the expected signature is not readily available, including a wide class of Gaussian and Markovian rough paths~\cite{CassLittererLyons13, CassOgrodnik17, FrizGess16}.
\end{remark}
  
\begin{remark} 
A series of papers~\cite{BCEF18, CEFMW14, EFMPW15} have studied the relation between It{\^o} and Stratonovich iterated integrals, particularity in relation with the quasi-shuffle algebra and Hoffman's exponential.
One of the main results of~\cite{BCEF18} is the existence of a unique a Hopf algebra morphism $\Psi^* : \HH^* \to \HH^*$ whose adjoint is the arborification of the Hoffman exponential, see~\cite[Thm.~2]{BCEF18}.
In particular, a level-$N$ branched rough path satisfying the shuffle identity can be mapped through $\Psi^*$ to a level-$N$ branched rough path satisfying the quasi-shuffle identity (e.g., Stratonovich- resp. It{\^o}-lift of Brownian motion), which provides in this case a higher order analogue of~\eqref{eq:Ito-Strat_conversion}.
Note that the results of this article are somewhat distinct from those of~\cite{BCEF18} since we are not directly interested in mapping one branched rough path to another, but rather in reinterpreting every branched rough path as a geometric rough path (over a different space).
\end{remark}

\appendix

\section{Symbolic index}

In this appendix, we collect the most used symbols of the article, together
with their meaning and the page where they were first introduced.

 \begin{center}

 \renewcommand{\arraystretch}{1.1}
 \begin{tabular}{lll}\toprule
 Symbol & Meaning & Page\\
 \midrule
 $|\sigma|$ & Number of nodes in a forest $\sigma \in \FF$ & \pageref{|tau| page ref}\\
 $\mathcal{A}$ & Algebra morphisms $M:\hat{\HH}\rightarrow \mathbf{L}(H_M)$ with $M(B) \subset \uu(H_M)$ & \pageref{mathcal{A} page ref}\\
$\mathcal{A}_k$ & $\{ (r_1,\ldots, r_m):r_i \in \{1,\ldots, k\}, m\geq 0 \}$ & \pageref{mathcal{A}_k page ref}\\ 
$\mathcal{A}^{\Pi}_{s}$ & $\{R\in \mathcal{A}_{k} : \deg_{\Pi}(R)\leq s\}$ & \pageref{mathcal{A}^{Pi}_s page ref}\\
 $B$ & Vector subspace of $\BB$ such that $\HH \simeq T(B)$ as Hopf algebras & \pageref{B page ref}\\ 
$\BB$ & Span of $\TT$ & \pageref{BB page ref}\\
$B_k$ & Subspace of $B$ spanned by its first $k$ basis elements & \pageref{B_k page ref}\\
$\mathcal{C}$ & The set of matrix coefficients of elements in $\mathcal{A}$ & \pageref{mathcal{C} page ref}\\ 
$\deg_{\Pi}(R)$ & $\sum_{j=1}^k\frac{n_j(R)}{p_j}$ if $\Pi=(p_1,\ldots ,p_k)$ & \pageref{deg_{Pi}(R) page ref}\\
$d_{p\var}$ & Homogeneous $p$-variation metric & \pageref{d_{pvar} page ref}\\
$\delta$ & Coproduct on the Grossman--Larson algebra & \pageref{delta page ref}\\
$\Delta$ & Coproduct on the Hopf algebra $T((V))$ & \pageref{Delta page ref}\\
$E_a(B)$ & $T(B)$ equipped with a universal locally $m$-convex topology & \pageref{E_a(B) page ref}\\
$E(B)$ & Completion of $E_a(B)$ & \pageref{E(B) page ref}\\
$G$ & Group-like elements in $E(B)$ & \pageref{G page ref}\\
$\GG^*$ & Group-like elements in $\HH^*$ & \pageref{GG^* page ref}\\
$\GG^N$ & Image of $\GG^*$ under $\rho^N$ & \pageref{GG^N page ref}\\
$\GG$ & $\GG^* \cap \hat \HH$ & \pageref{GG page ref}\\
$\GG_k$ & $\GG\cap P(B_k)$ & \pageref{GG_k page ref}\\
$(\gamma_k)_{k\geq 0}$ & Sequence of semi-norms on $B$ & \pageref{gamma_k page ref} \\
$G^{(\Pi,1)}(V)$ & Exponential of Lie algebra in $T^{(\Pi,1)}(V)$ generated by $V$ & \pageref{G^{(Pi,1)}(V) page ref}\\
 \bottomrule
\end{tabular}\newpage
\begin{tabular}{lll}\toprule
Symbol & Meaning & Page\\
\midrule
$\FF$ & Set of all forests & \pageref{FF page ref}\\
 $\HH$ & Span of $\FF$ equipped with Grossman--Larson Hopf algebra & \pageref{HH page ref}\\
$\HH^*$ & Space of formal series in forests & \pageref{HH^* page ref}\\
$\HH^N$ & Subspace of $\HH^*$ spanned by forests with at most $N$ &\pageref{HH^N page ref}\\
$\HH^{(N)}$ & Subspace of $\HH^*$ spanned by forests with exactly $N$ nodes & \pageref{HH^(N) page ref}\\
$ \hat \HH$ & Completion of $\HH$ under topology induced by $\HH \simeq E_a(B)$ & \pageref{hat HH page ref}\\
$N_{\tau}$ & Node set of a rooted tree $\tau$ & \pageref{N_{tau} page ref}\\
$n_j(R)$ & $|\{i:r_i=j\}|$ if $R=(r_1,\ldots, r_m) \in \A_k$ & \pageref{n_j(R) page ref}\\
$\Vert \cdot \Vert_{p\var}$ & The $p$-variation norm & \pageref{Vert cdot Vert_{pvar} page ref}\\
$\omega$ & A control & \pageref{omega page ref}\\
$\pi^N$ & Projection of $\HH^*$ onto $\HH^{(N)}$ & \pageref{pi^N page ref}\\
$\pi_R$ & Projection of $T((V))$ onto $V^{\otimes R}$ & \pageref{pi_R page ref}\\
$P(V)$ & $\prod_{m=0}^\infty V^{\hat{\otimes} m}$ & \pageref{P(V) page ref}\\
$\Psi$ & Algebra isomorphism from $\HH^{\lfloor p \rfloor}$ to $T^{(\Pi,1)}(B_k)$ & \pageref{Psi page ref}\\
$\rho^N$ & Projection of $\HH^*$ onto $\HH^N$ & \pageref{rho page ref}\\ 
$\rho_{\Pi-var}$ & The $\Pi$-variation metric & \pageref{rho_{Pi-var} page ref}\\
$\curvearrowright$ & A pre-Lie product $\BB \times \BB \rightarrow \BB$ & \pageref{curvearrowright page ref}\\
$\star$ & Product in the Grossman--Larson algebra & \pageref{star page ref}\\
$\triangle$ & $\{(s,t): 0\leq s\leq t\leq 1\}$ & \pageref{triangle page ref}\\ 
$S(X)$ & Signature of a rough path $X$ & \pageref{S(X) page ref}\\ 
$\TT$ & Set of all rooted trees with node label set $\{1,\ldots, d\}$ & \pageref{TT page ref}\\
$T((V))$ & $\prod_{m=0}^\infty V^{{\otimes} m}$, i.e., formal series of tensors of $V$ & \pageref{T((V)) page ref}\\
$T(V)$ & Finite series in $T((V))$ & \pageref{T(V) page ref}\\
$T^{(\Pi,s)}(V)$ & Inhomogeneous tensors with $\Pi$-degree at most $s$ & \pageref{T^{(Pi,s)}(V) page ref}\\
$\dot \otimes $ & Tensor product on $T((V))$ & \pageref{dot otimes page ref}\\
$ \otimes$ & Algebraic tensor product & \pageref{otimes page ref}\\
$\hat \otimes$ & Completion of the algebraic tensor product & \pageref{hat otimes page ref}\\
$\uu(H)$ & Anti-Hermitian operators on $H$ &\pageref{uu(H) page ref}\\
$V^{\otimes R}$ & $V^{r_1}\otimes \ldots \otimes V^{r_m}$ if $R=(r_1,\ldots ,r_m)\in \A_k$ & \pageref{V^{otimes R} page ref}\\
$\bar{X}$ & Image under $\Psi$ of a branched rough path $X$ & \pageref{bar{X} page ref}\\

 \bottomrule
 \end{tabular}
 \end{center}

\bibliographystyle{./Martin}
\bibliography{./AllRefs}

\newcommand{\etalchar}[1]{$^{#1}$}
\begin{thebibliography}{CEFMW14}
\expandafter\ifx\csname url\endcsname\relax
  \def\url#1{\texttt{#1}}\fi
\expandafter\ifx\csname urlprefix\endcsname\relax\def\urlprefix{URL }\fi
\expandafter\ifx\csname href\endcsname\relax
  \def\href#1#2{#2}\fi
\expandafter\ifx\csname burlalt\endcsname\relax
  \def\burlalt#1#2{\href{#2}{\texttt{#1}}}\fi

\bibitem[AcG80]{AG81}
\textsc{A.~A. Agra\v~cev} and \textsc{R.~V. Gamkrelidze}.
\newblock Chronological algebras and nonstationary vector fields.
\newblock In \emph{Problems in geometry, {V}ol. 11 ({R}ussian)},  135--176,
  243. Akad. Nauk SSSR, Vsesoyuz. Inst. Nauchn. i Tekhn. Informatsii, Moscow,
  1980.

\bibitem[Aze82]{Azencott82}
\textsc{R.~Azencott}.
\newblock Formule de {T}aylor stochastique et d\'eveloppement asymptotique
  d'int\'egrales de {F}eynman.
\newblock In \emph{Seminar on {P}robability, {XVI}, {S}upplement}, vol. 921 of
  \emph{Lecture Notes in Math.},  237--285. Springer, Berlin-New York, 1982.

\bibitem[BA89]{BenArous89}
\textsc{G.~Ben~Arous}.
\newblock Flots et s\'eries de {T}aylor stochastiques.
\newblock \emph{Probab. Theory Related Fields} \textbf{81}, no.~1, (1989),
  29--77.

\bibitem[BCCH17]{BCCH17}
\textsc{Y.~{Bruned}}, \textsc{A.~{Chandra}}, \textsc{I.~{Chevyrev}}, and
  \textsc{M.~{Hairer}}.
\newblock {Renormalising SPDEs in regularity structures}.
\newblock \emph{ArXiv e-prints} (2017).
\newblock \burlalt{arXiv:1711.10239}{http://arxiv.org/abs/1711.10239}.

\bibitem[BCE18]{BCEF18}
\textsc{Y.~{Bruned}}, \textsc{C.~{Curry}}, and \textsc{K.~{Ebrahimi-Fard}}.
\newblock {Quasi-shuffle algebras and renormalisation of rough differential
  equations}.
\newblock \emph{ArXiv e-prints} (2018).
\newblock \burlalt{arXiv:1801.02964}{http://arxiv.org/abs/1801.02964}.

\bibitem[BCFP17]{BCFP17}
\textsc{Y.~{Bruned}}, \textsc{I.~{Chevyrev}}, \textsc{P.~K. {Friz}}, and
  \textsc{R.~{Preiss}}.
\newblock {A Rough Path Perspective on Renormalization}.
\newblock \emph{ArXiv e-prints} (2017).
\newblock \burlalt{arXiv:1701.01152}{http://arxiv.org/abs/1701.01152}.

\bibitem[BFG{\etalchar{+}}17]{BFGMS17}
\textsc{C.~{Bayer}}, \textsc{P.~K. {Friz}}, \textsc{P.~{Gassiat}},
  \textsc{J.~{Martin}}, and \textsc{B.~{Stemper}}.
\newblock {A regularity structure for rough volatility}.
\newblock \emph{ArXiv e-prints} (2017).
\newblock \burlalt{arXiv:1710.07481}{http://arxiv.org/abs/1710.07481}.

\bibitem[BGLY16]{Boedihardjo16}
\textsc{H.~Boedihardjo}, \textsc{X.~Geng}, \textsc{T.~Lyons}, and
  \textsc{D.~Yang}.
\newblock The signature of a rough path: uniqueness.
\newblock \emph{Adv. Math.} \textbf{293}, (2016), 720--737.
\newblock
  \burlalt{doi:10.1016/j.aim.2016.02.011}{http://dx.doi.org/10.1016/j.aim.2016.02.011}.

\bibitem[BHZ16]{BHZ16}
\textsc{Y.~{Bruned}}, \textsc{M.~{Hairer}}, and \textsc{L.~{Zambotti}}.
\newblock {Algebraic renormalisation of regularity structures}.
\newblock \emph{ArXiv e-prints} (2016).
\newblock \burlalt{arXiv:1610.08468}{http://arxiv.org/abs/1610.08468}.

\bibitem[{Boe}15]{Boedihardjo15}
\textsc{H.~{Boedihardjo}}.
\newblock {Decay Rate of Iterated Integrals of Branched Rough Paths}.
\newblock \emph{ArXiv e-prints} (2015).
\newblock To appear in Ann. Inst. H. Poincaré Anal. Non Linéaire.
\newblock \burlalt{arXiv:1501.05641}{http://arxiv.org/abs/1501.05641}.
\newblock
  \burlalt{doi:10.1016/j.anihpc.2017.09.002}{http://dx.doi.org/10.1016/j.anihpc.2017.09.002}.

\bibitem[But16]{Butcher16}
\textsc{J.~C. Butcher}.
\newblock \emph{Numerical methods for ordinary differential equations}.
\newblock John Wiley \& Sons, Ltd., Chichester, third ed., 2016,  xxiii+513.
\newblock With a foreword by J. M. Sanz-Serna.

\bibitem[CEFMW14]{CEFMW14}
\textsc{C.~Curry}, \textsc{K.~Ebrahimi-Fard}, \textsc{S.~J.~A. Malham}, and
  \textsc{A.~Wiese}.
\newblock L\'evy processes and quasi-shuffle algebras.
\newblock \emph{Stochastics} \textbf{86}, no.~4, (2014), 632--642.
\newblock
  \burlalt{doi:10.1080/17442508.2013.865131}{http://dx.doi.org/10.1080/17442508.2013.865131}.

\bibitem[Cha10]{Chapoton10}
\textsc{F.~Chapoton}.
\newblock Free pre-{L}ie algebras are free as {L}ie algebras.
\newblock \emph{Canad. Math. Bull.} \textbf{53}, no.~3, (2010), 425--437.
\newblock
  \burlalt{doi:10.4153/CMB-2010-063-2}{http://dx.doi.org/10.4153/CMB-2010-063-2}.

\bibitem[Che57]{Chen57}
\textsc{K.-T. Chen}.
\newblock Integration of paths, geometric invariants and a generalized
  {B}aker-{H}ausdorff formula.
\newblock \emph{Ann. of Math. (2)} \textbf{65}, (1957), 163--178.

\bibitem[Che18]{Chevyrev18}
\textsc{I.~Chevyrev}.
\newblock Random walks and {L}\'evy processes as rough paths.
\newblock \emph{Probab. Theory Related Fields} \textbf{170}, no. 3-4, (2018),
  891--932.
\newblock
  \burlalt{doi:10.1007/s00440-017-0781-1}{http://dx.doi.org/10.1007/s00440-017-0781-1}.

\bibitem[CL01]{Chapoton01}
\textsc{F.~Chapoton} and \textsc{M.~Livernet}.
\newblock Pre-{L}ie algebras and the rooted trees operad.
\newblock \emph{Internat. Math. Res. Notices} , no.~8, (2001), 395--408.
\newblock
  \burlalt{doi:10.1155/S1073792801000198}{http://dx.doi.org/10.1155/S1073792801000198}.

\bibitem[CL16]{ChevyrevLyons16}
\textsc{I.~Chevyrev} and \textsc{T.~Lyons}.
\newblock Characteristic functions of measures on geometric rough paths.
\newblock \emph{Ann. Probab.} \textbf{44}, no.~6, (2016), 4049--4082.
\newblock
  \burlalt{doi:10.1214/15-AOP1068}{http://dx.doi.org/10.1214/15-AOP1068}.

\bibitem[CLL13]{CassLittererLyons13}
\textsc{T.~Cass}, \textsc{C.~Litterer}, and \textsc{T.~Lyons}.
\newblock Integrability and tail estimates for {G}aussian rough differential
  equations.
\newblock \emph{Ann. Probab.} \textbf{41}, no.~4, (2013), 3026--3050.

\bibitem[CO17]{CassOgrodnik17}
\textsc{T.~Cass} and \textsc{M.~Ogrodnik}.
\newblock Tail estimates for {M}arkovian rough paths.
\newblock \emph{Ann. Probab.} \textbf{45}, no.~4, (2017), 2477--2504.

\bibitem[CW16]{CassWeidner16}
\textsc{T.~{Cass}} and \textsc{M.~P. {Weidner}}.
\newblock {Tree algebras over topological vector spaces in rough path theory}.
\newblock \emph{ArXiv e-prints} (2016).
\newblock \burlalt{arXiv:1604.07352}{http://arxiv.org/abs/1604.07352}.

\bibitem[EFMPW15]{EFMPW15}
\textsc{K.~Ebrahimi-Fard}, \textsc{S.~J.~A. Malham}, \textsc{F.~Patras}, and
  \textsc{A.~Wiese}.
\newblock The exponential {L}ie series for continuous semimartingales.
\newblock \emph{Proc. A.} \textbf{471}, no. 2184, (2015), 20150429, 19.
\newblock
  \burlalt{doi:10.1098/rspa.2015.0429}{http://dx.doi.org/10.1098/rspa.2015.0429}.

\bibitem[FGGR16]{FrizGess16}
\textsc{P.~K. Friz}, \textsc{B.~Gess}, \textsc{A.~Gulisashvili}, and
  \textsc{S.~Riedel}.
\newblock The {J}ain--{M}onrad criterion for rough paths and applications to
  random {F}ourier series and non-{M}arkovian {H}\"ormander theory.
\newblock \emph{Ann. Probab.} \textbf{44}, no.~1, (2016), 684--738.
\newblock \burlalt{doi:10.1214/14-AOP986}{http://dx.doi.org/10.1214/14-AOP986}.

\bibitem[Foi02]{Foissy02}
\textsc{L.~Foissy}.
\newblock Finite-dimensional comodules over the {H}opf algebra of rooted trees.
\newblock \emph{J. Algebra} \textbf{255}, no.~1, (2002), 89--120.
\newblock
  \burlalt{doi:10.1016/S0021-8693(02)00110-2}{http://dx.doi.org/10.1016/S0021-8693(02)00110-2}.

\bibitem[FS82]{FollandStein82}
\textsc{G.~B. Folland} and \textsc{E.~M. Stein}.
\newblock \emph{Hardy spaces on homogeneous groups}, vol.~28 of
  \emph{Mathematical Notes}.
\newblock Princeton University Press, Princeton, N.J.; University of Tokyo
  Press, Tokyo, 1982,  xii+285.

\bibitem[FS17]{FrizShekhar17}
\textsc{P.~Friz} and \textsc{A.~Shekhar}.
\newblock {G}eneral rough integration, {L}{\'e}vy rough paths and a
  {L}{\'e}vy--{K}intchine type formula.
\newblock \emph{Ann. Probab.} \textbf{45}, no.~4, (2017), 2707–2765.
\newblock
  \burlalt{doi:10.1214/16-AOP1123}{http://dx.doi.org/10.1214/16-AOP1123}.

\bibitem[FV08]{FV08}
\textsc{P.~Friz} and \textsc{N.~Victoir}.
\newblock Euler estimates for rough differential equations.
\newblock \emph{J. Differential Equations} \textbf{244}, no.~2, (2008),
  388--412.

\bibitem[FV10]{FrizVictoir10}
\textsc{P.~K. Friz} and \textsc{N.~B. Victoir}.
\newblock \emph{Multidimensional stochastic processes as rough paths}, vol. 120
  of \emph{Cambridge Studies in Advanced Mathematics}.
\newblock Cambridge University Press, Cambridge, 2010,  xiv+656.

\bibitem[Gub10]{Gubinelli10}
\textsc{M.~Gubinelli}.
\newblock Ramification of rough paths.
\newblock \emph{J. Differential Equations} \textbf{248}, no.~4, (2010),
  693--721.
\newblock
  \burlalt{doi:10.1016/j.jde.2009.11.015}{http://dx.doi.org/10.1016/j.jde.2009.11.015}.

\bibitem[GV95]{Giambruno95}
\textsc{A.~Giambruno} and \textsc{A.~Valenti}.
\newblock On minimal {$*$}-identities of matrices.
\newblock \emph{Linear and Multilinear Algebra} \textbf{39}, no.~4, (1995),
  309--323.
\newblock
  \burlalt{doi:10.1080/03081089508818405}{http://dx.doi.org/10.1080/03081089508818405}.

\bibitem[Gyu08]{Gyurko08}
\textsc{L.~G. Gyurk\'o}.
\newblock \emph{Numerical methods for approximating solutions to rough
  differential equations}.
\newblock Ph.D. thesis, University of Oxford, 2008.

\bibitem[Gyu16]{Gyurko16}
\textsc{L.~G. Gyurk\'o}.
\newblock Differential equations driven by {$\Pi$}-rough paths.
\newblock \emph{Proc. Edinb. Math. Soc. (2)} \textbf{59}, no.~3, (2016),
  741--758.

\bibitem[Hai14]{Hairer14}
\textsc{M.~Hairer}.
\newblock A theory of regularity structures.
\newblock \emph{Invent. Math.} \textbf{198}, no.~2, (2014), 269--504.
\newblock
  \burlalt{doi:10.1007/s00222-014-0505-4}{http://dx.doi.org/10.1007/s00222-014-0505-4}.

\bibitem[HK15]{HairerKelly15}
\textsc{M.~Hairer} and \textsc{D.~Kelly}.
\newblock Geometric versus non-geometric rough paths.
\newblock \emph{Ann. Inst. Henri Poincar\'e Probab. Stat.} \textbf{51}, no.~1,
  (2015), 207--251.
\newblock
  \burlalt{doi:10.1214/13-AIHP564}{http://dx.doi.org/10.1214/13-AIHP564}.

\bibitem[HL10]{Hambly10}
\textsc{B.~Hambly} and \textsc{T.~Lyons}.
\newblock Uniqueness for the signature of a path of bounded variation and the
  reduced path group.
\newblock \emph{Ann. of Math. (2)} \textbf{171}, no.~1, (2010), 109--167.
\newblock
  \burlalt{doi:10.4007/annals.2010.171.109}{http://dx.doi.org/10.4007/annals.2010.171.109}.

\bibitem[HS90]{HebischSikora90}
\textsc{W.~Hebisch} and \textsc{A.~Sikora}.
\newblock A smooth subadditive homogeneous norm on a homogeneous group.
\newblock \emph{Studia Math.} \textbf{96}, no.~3, (1990), 231--236.

\bibitem[LV04]{LyonsVictoir}
\textsc{T.~Lyons} and \textsc{N.~Victoir}.
\newblock Cubature on {W}iener space.
\newblock \emph{Proc. R. Soc. Lond. Ser. A Math. Phys. Eng. Sci.} \textbf{460},
  no. 2041, (2004), 169--198.
\newblock Stochastic analysis with applications to mathematical finance.
\newblock
  \burlalt{doi:10.1098/rspa.2003.1239}{http://dx.doi.org/10.1098/rspa.2003.1239}.

\bibitem[LV06]{LV06}
\textsc{A.~Lejay} and \textsc{N.~Victoir}.
\newblock On {$(p,q)$}-rough paths.
\newblock \emph{J. Differential Equations} \textbf{225}, no.~1, (2006),
  103--133.

\bibitem[LV07]{LyonsVictoir07}
\textsc{T.~Lyons} and \textsc{N.~Victoir}.
\newblock An extension theorem to rough paths.
\newblock \emph{Ann. Inst. H. Poincar\'e Anal. Non Lin\'eaire} \textbf{24},
  no.~5, (2007), 835--847.

\bibitem[Lyo98]{Lyons98}
\textsc{T.~J. Lyons}.
\newblock Differential equations driven by rough signals.
\newblock \emph{Rev. Mat. Iberoamericana} \textbf{14}, no.~2, (1998), 215--310.
\newblock \burlalt{doi:10.4171/RMI/240}{http://dx.doi.org/10.4171/RMI/240}.

\end{thebibliography}

\end{document}